\documentclass[a4paper,11pt]{article}

\usepackage{amsmath,amsthm,amssymb,mathrsfs,latexsym,amsfonts}
\usepackage{graphicx,psfrag,epsfig}
\usepackage[english]{babel}
\usepackage[latin1]{inputenc}

\newtheorem{theorem}{Theorem}[section]
\newtheorem{lemma}[theorem]{Lemma}
\newtheorem{proposition}[theorem]{Proposition}

\newtheorem{definition}[theorem]{Definition}
\newtheorem{remark}[theorem]{Remark}

\newcounter{paraga}[section]
\renewcommand{\theparaga}{{\bf\arabic{paraga}.}}
\newcommand{\paraga}{\medskip \addtocounter{paraga}{1} 
\noindent{\theparaga\ } }

\begin{document}

\bibliographystyle{amsalpha}

\def\MP{\,{<\hspace{-.5em}\cdot}\,}
\def\SP{\,{>\hspace{-.3em}\cdot}\,}
\def\PM{\,{\cdot\hspace{-.3em}<}\,}
\def\PS{\,{\cdot\hspace{-.3em}>}\,}
\def\EP{\,{=\hspace{-.2em}\cdot}\,}
\def\PP{\,{+\hspace{-.1em}\cdot}\,}
\def\PE{\,{\cdot\hspace{-.2em}=}\,}
\def\N{\mathbb N}
\def\C{\mathbb C}
\def\Q{\mathbb Q}
\def\R{\mathbb R}
\def\T{\mathbb T}
\def\A{\mathbb A}
\def\Z{\mathbb Z}
\def\demi{\frac{1}{2}}

\begin{titlepage}
\author{Abed Bounemoura~\footnote{Laboratoire Mathématiques d'Orsay et Institut Mathématiques de Jussieu}}
\title{\LARGE{\textbf{Generic super-exponential stability of invariant tori in Hamiltonian systems.}}}
\end{titlepage}

\maketitle

\begin{abstract}
In this article, we consider solutions starting close to some linearly stable invariant tori in an analytic Hamiltonian system and we prove results of stability for a super-exponentially long interval of time, under generic conditions. The proof combines classical Birkhoff normal forms and a new method to obtain generic Nekhoroshev's estimates developed by the author and L. Niederman in another paper. We will mainly focus on the neighbourhood of elliptic fixed points, since with our approach the other cases are completely similar.
\end{abstract}
  
\section{Introduction and main results}

In this paper, we are interested in the stability properties of some linearly stable invariant tori in analytic Hamiltonian systems. Let us begin by the case of elliptic fixed points.

\paraga As the problem is local, it is enough to consider a Hamiltonian $H$ defined and analytic on an open neighbourhood of $0$ in $\R^{2n}$, having the origin as a fixed point. Up to an irrelevant additive constant and expanding the Hamiltonian as a power series at the origin, we can write
\[ H(z)=H_2(z)+V(z), \]
where $z$ is sufficiently close to $0$ in $\R^{2n}$, $H_2$ is the quadratic part of $H$ at $0$ and $V(z)=O(||z||^3)$. Recall that the fixed point is said to be elliptic if the spectrum of the linearized system is purely imaginary. Then it has the form $\{\pm i\alpha_1, \dots, \pm i\alpha_n\}$, for some vector $\alpha=(\alpha_1, \dots, \alpha_n)\in\R^n$ which is called the normal (or characteristic) frequency. Due to the symplectic character of the equations, such equilibria are the only linearly stable fixed points. Now we assume that the components of $\alpha$ are all distinct so that we can make a symplectic linear change of variables that diagonalizes the quadratic part:
\[ H(z)=\sum_{i=1}^{n}\frac{\alpha_i}{2}(z_i^2+z_{n+i}^2)+V(z)=\alpha.\tilde{I}+V(z), \]
where $\tilde{I}=\tilde{I}(z)$ is the vector of ``formal actions", that is
\[ \tilde{I}(z)=\frac{1}{2}(z_1^2+z_{n+1}^2, \dots, z_n^2+z_{2n}^2) \in \R^n. \]
Assuming the components of $\alpha$ are all of the same sign, it is easy to see that $H$ is a Lyapunov function so the fixed point is stable. But in the general case, one has to study the influence of the higher order terms $V(z)$, and we will explain how it can be done using classical perturbation theory.

\paraga Let us first note that, given a solution $z(t)$ of $H$, if $\tilde{I}(t)=\tilde{I}(z(t))$ then 
\[ |\tilde{I}(t)|_1=\sum_{i=1}^{n}|\tilde{I}_i(t)| \] 
is (up to a factor one-half) the square of the Euclidean distance of $z(t)$ to the origin, so that Lyapunov stability can be proved if $|\tilde{I}(t)-\tilde{I}(0)|_1$ does not vary much for all times.

Now in order to study the dynamics on a small neighbourhood of size $\rho$ around the origin in $\R^{2n}$, it is more convenient to change coordinates by performing the standard scalings
\[ z \longmapsto \rho z, \quad H \longmapsto \rho^{-2}H, \]
to have a Hamiltonian defined on a fixed neighbourhood of zero in $\R^{2n}$. Then, by analyticity, we extend the resulting Hamiltonian to a holomorphic function on some complex neighbourhood of zero in $\C^{2n}$. So eventually we will consider the following setting: we define the Euclidean ball in $\C^{2n}$
\[ \mathcal{D}_s=\{z\in\C^{2n}\;|\;||z||<s\}\] 
of radius $s$ around the origin, and if $\mathcal{A}_s$ is the space of holomorphic functions on $\mathcal{D}_s$ which are real valued for real arguments, endowed with its usual supremum norm $|\,.\,|_s$, we consider
\begin{equation}\label{Ham1}
\begin{cases}
H(z)=\alpha.\tilde{I}+f(z) \\
H \in \mathcal{A}_s, |f|_s < \rho.
\end{cases}\tag{$A$} 
\end{equation} 
Let us emphasize that the small parameter $\rho$, which was originally describing the size of the neighbourhood of $0$, now describes the size of the ``perturbation" $f$ on a neighbourhood of \emph{fixed size} $s$. Without loss of generality, we may assume $s>3$.

\paraga Probably the main tool to investigate stability properties is the construction of normal forms using averaging methods, and in this case these are the so-called Birkhoff normal forms. For an integer $m\geq 1$, assuming $\alpha$ is non-resonant up to order $2m$, that is  
\[ k.\alpha \neq 0, \quad k\in\Z^n, \; 0<|k|_1\leq 2m,  \]
there exists an analytic symplectic transformation $\Phi_m$ close to identity such that $H\circ\Phi_m$ is in Birkhoff normal form up to order $2m$, that is
\[ H\circ\Phi_m(z)=h_m(\tilde{I})+f_m(z), \]
where $h_m$ is a polynomial of degree at most $m$ in the $\tilde{I}$ variables, and the remainder $f_m$ is roughly of order $\rho^{2m-1}$ (since before the scaling $f_m(z)$ is of order $\Vert z\Vert^{2m+1}$, see \cite{Bir66}, or \cite{Dou88} for a more recent exposition). The polynomials $h_m$ are uniquely defined once $\alpha$ is fixed, and are usually called the Birkhoff invariants. Therefore the transformed Hamiltonian is the sum of an integrable part $h_m$, for which the origin is trivially stable, since $\tilde{I}(t)$ is constant for all times, and a smaller perturbation $f_m$. Moreover, if $\alpha$ is non-resonant up to any order, we can even define a formal symplectic transformation $\Phi_\infty$ and a formal power series $h_\infty=\sum_{k\geq 1}h^k$, with $h_m=\sum_{k=1}^{m}h^k$, such that 
\begin{equation*}
H\circ\Phi_\infty(z)=h_\infty(\tilde{I}).
\end{equation*} 
In general the series $h_\infty$ is divergent (this is a result of Siegel) and the convergence properties of the transformation $\Phi_\infty$ are even more subtle (see \cite{PM03}). However, Birkhoff normal forms at finite order are still very useful, not only because the ``perturbation" $f_m$ is made smaller, but also because the ``integrable" part $h_m$, for $m\geq 2$, is now non-linear and other classical techniques from perturbation theory can be used. 

\paraga First, in the case $n=2$, a complete result of stability follows from KAM theory. Indeed, if the frequency $\alpha\in\R^2$ is non-resonant up to order $4$, the Birkhoff normal form reads 
\[ H(z)=\alpha.\tilde{I}+\beta \tilde{I}.\tilde{I}+f_2(z),\]
with $\beta$ a symmetric matrix of size $n=2$ and $f_2$ a small perturbation. This time we consider the non-linear part $h_2(\tilde{I})=\alpha.\tilde{I}+\beta \tilde{I}.\tilde{I}$ as the integrable system, and if it is isoenergetically non degenerate, the persistence of two-dimensional tori in each energy level close to the fixed point implies Lyapunov stability (see \cite{AKN97}, \cite{Arn63b}, or \cite{Arn61} for other results).  

However, for $n\geq 3$, it is believed that ``generic" elliptic fixed points are unstable, although this is totally unclear for the moment (see \cite{DLC83}, \cite{Dou88} and \cite{KMV04}).    

Therefore, for $n\geq 3$, stability results under general assumptions can only concern finite but hopefully long intervals of time, and this is the content of the paper. More precisely, we will prove, under generic assumptions and provided $\rho$ is sufficiently small, that for all initial conditions the variation $|\tilde{I}(t)-\tilde{I}(0)|_1$ is of order $\rho$ for $t\in T(\rho)$, where $T(\rho)$ is an interval of time of order $\exp\left(\exp(\rho^{-1})\right)$ (see Theorem~\ref{thelli} for a precise formulation). The interpretation in the original coordinates is the following: if a solution starts in a sufficiently small neighbourhood of the origin, it stays in some larger neighbourhood during an interval of time which is super-exponentially long with respect to the inverse of the initial distance to the origin. But first, let us describe previously known results on exponential stability, where there were basically two strategies.

\paraga In a first approach, one assumes a Diophantine condition on $\alpha$, that is there exist $\gamma>0$ and $\tau> n-1$ such that
\begin{equation*}\label{Dioph}
|k.\alpha| \geq \gamma |k|_{1}^{-\tau}, \quad k\in\Z^n\setminus\{0\},  
\end{equation*}
but no conditions on the Birkhoff invariants. From the point of view of perturbation theory, the linear part is considered as the integrable system. In particular, $\alpha$ is non-resonant up to any order, hence we can perform any finite number of Birkhoff normalizations, and since we have a control on the small divisors, we can precisely estimate the size of the remainder $f_m$ (in terms of $\gamma$ and $\tau$). The usual trick is then to optimize the choice of $m$ as a function of $\rho$ in order to obtain an exponentially small remainder with respect to $\rho^{-1}$. Therefore the exponential stability is immediately read from the normal form, and this requires only an assumption on the linear part (see \cite{GDFGS} or \cite{DG96b}). The above Diophantine condition has full Lebesgue measure. However, as we will see later, the threshold of the perturbation and the constants of stability are very sensitive to the Diophantine properties of $\alpha$, in particular the small parameter $\gamma$. 

\paraga The second approach is fundamentally different, and it does not rely on the arithmetic properties of $\alpha$. Here, one just assumes that $\alpha$ is non-resonant up to order $4$, so that the Hamiltonian reduces to   
\[ H(z)=\alpha.\tilde{I}+\beta \tilde{I}.\tilde{I}+f_2(z).\]
In this case $h_2(\tilde{I})=\alpha.\tilde{I}+\beta \tilde{I}.\tilde{I}$ is considered as the integrable system ($\beta$ being a symmetric matrix of size $n$). Now we suppose that the non-linear part is convex, which is equivalent to $\beta$ being sign definite. Under those assumptions, it was predicted and partially proved by Lochak (\cite{Loc92} and \cite{Loc95}), and completely proved independently by Niederman (\cite{Nie98}) and Fassò, Guzzo and Benettin (\cite{BFG98} and \cite{BFG98b}) that exponential stability holds. Their proofs are based on the implementation of Nekhoroshev's estimates in Cartesian coordinates, but they are radically different: the first one uses Lochak's method of periodic averagings and simultaneous Diophantine approximations, while the second one is based on Nekhoroshev's original mechanism. The proof of Niederman was later clarified by Pöschel (\cite{Pos99a}). However, the method of Lochak was restricted to the convex case, and it was not clear how to remove this hypothesis to have a result valid in a more general context.

\paraga In this paper, using the method of \cite{BN09} we are able to replace the convexity condition by a generic assumption. Then, combining both Birkhoff theory and Nekhoroshev theory as in \cite{MG95}, we will obtain the following result.

\begin{theorem}\label{thelli}
Suppose $H$ is as in~(\ref{Ham1}), with $\alpha$ non-resonant up to any order. Then under a generic condition $(G)$ on $h_\infty$, there exist positive constants $a,a',c_1,c_2$ and $\rho_0$ such that for $\rho\leq\rho_0$, every solution $z(t)$ of $H$ with $|\tilde{I}(0)|_1<1$ satisfies
\[ |\tilde{I}(t)-\tilde{I}(0)|_1< c_1\rho,\quad |t| < \exp\left(\rho^{-a'}\exp(c_2a'\rho^{-a})\right).  \]
\end{theorem}

Denoting $h_\infty=\sum_{k\geq 1}h^k$ and $h_m=\sum_{k=1}^{m}h^k$, let us explain our generic condition $(G)$ on the formal power series $h_\infty$. In fact
\[ (G)=\bigcup_{m\in \N^{*}}(G_m) \]
consists in countably many conditions, where $(G_m)$ is a condition on $h_m$. The first condition $(G_1)$ requires that $h_1(I)=\alpha.I$ with a $(\gamma,\tau)$-Diophantine vector $\alpha$. The other conditions $(G_m)$, for $m\geq 2$, are that each polynomial function $h_m$ belongs to a special class of functions called $SDM_{\gamma'}^{\tau'}$ which was introduced in \cite{BN09} (SDM stands for ``Simultaneous Diophantine Morse" functions, see Appendix~\ref{generic} for a definition). In this appendix we will show that each condition $(G_m)$ is of full Lebesgue measure in the finite dimensional space of polynomials of degree $m$ with $n$ variables, assuming $\tau$ and $\tau'$ are large enough. This is well-known for $m=1$, it will be elementary for $m=2$ (see Theorem~\ref{genericquadratic}) but for $m>2$ it requires the quantitative Morse-Sard theory of Yomdin (\cite{Yom83}, \cite{YC04}, see Theorem~\ref{propreva} in the appendix). Let us point out that this would have not been possible if we had assumed $h_m$, for $m\geq2$, to be steep in the sense of Nekhoroshev, as polynomials are generically steep only if their degrees are sufficiently large with respect to the number of degrees of freedom (see \cite{LM88}). 

Our condition $(G)$ on the formal series $h_\infty$ is therefore of ``full Lebesgue measure at any order". From an abstract point of view, this condition defines a prevalent set in the space of formal power series, where prevalence is an analog of the notion of full Lebesgue measure in the context of infinite dimensional vector spaces. This will be proved in Appendix~\ref{generic}, Theorem~\ref{genericseries}. In Theorem~\ref{thelli}, we can choose the exponents
\[ a=(1+\tau)^{-1}, \; a'=3^{-1}(2(n+1)\tau' )^{-n}, \]
and our threshold $\rho_0$ depends in particular on $\gamma$ and $\gamma'$. Moreover our constants $c_1$ and $c_2$ also depend on $\gamma$ but not on $\gamma'$, and we shall be a little more precise later on.
 
As we have already explained, the proof is based on a combination of Birkhoff normalizations up to an exponentially small remainder, which are well-known (a statement is recalled in Proposition~\ref{B1} below), and Nekhoroshev's estimates for a generic integrable Hamiltonian near an elliptic fixed point (Theorem~\ref{N1} below). The latter result is new, and it will follow rather easily from the new approach of Nekhoroshev theory in a generic case taken in \cite{BN09}.

\paraga As a direct consequence of our Nekhoroshev estimates near an elliptic fixed point, we can derive an exponential stability result more general than those obtained in \cite{BFG98} and \cite{Nie98}. Like in those papers, we only require $\alpha$ to be non-resonant up to order $4$, and after the scalings 
\[ z \longmapsto \rho z, \quad H \longmapsto \rho^{-4}H, \quad \alpha \longmapsto \rho^{2}\alpha, \]
we consider
\begin{equation}\label{Ham2}
\begin{cases}
H(z)=\alpha.\tilde{I}+\beta \tilde{I}.\tilde{I}+f(z) \\
H \in \mathcal{A}_s, |f|_s < \rho. 
\end{cases}\tag{$B$} 
\end{equation} 
However, instead of assuming that $\beta$ is sign definite, our result applies to Lebesgue almost all symmetric matrices $\beta$ without any condition on $\alpha$. Let $S_n(\R)$ be the space of symmetric matrices of size $n$ with real entries. 

\begin{theorem}\label{thelli2}
Suppose $H$ is as in~(\ref{Ham2}). For Lebesgue almost all $\beta\in S_n(\R)$, there exist positive constants $a',b'$ and $\rho_0$ such that, for $\rho\leq\rho_0$, every solution $z(t)$ of $H$ with $|\tilde{I}(0)|_1<1$ satisfies
\[|\tilde{I}(t)-\tilde{I}(0)|_1<n(n^2+1)\rho^{-b'},\quad |t|<\exp(\rho^{-a'}).\]
\end{theorem}

The above theorem is a direct consequence of Theorem~\ref{N1} below, provided that $h_2(\tilde{I})=\alpha.\tilde{I}+\beta \tilde{I}.\tilde{I}$ belongs to $SDM_{\gamma'}^{\tau'}$. But we will prove in Appendix~\ref{generic} that this happens for almost all symmetric matrices $\beta$, independently of $\alpha$ (see Theorem~\ref{genericquadratic}). Once again, let us also mention that this result is not possible in the steep case, as the quadratic part $h^2(\tilde{I})=\beta\tilde{I}.\tilde{I}$ is steep if and only if $\beta$ is sign definite. In the above statement one can choose the exponents
\[ a'=b'=3^{-1}(2(n+1)\tau' )^{-n}, \]
and the threshold $\rho_0$ depends on $\gamma'$. 

\paraga Let us add that in order to avoid useless expressions, we will only keep track of the small parameters $\rho$, $\gamma$ and $\gamma'$ and replace any other positive constants by a dot ($\cdot$) when it is convenient.

Moreover, in this text we shall use various norms for vectors $v\in\R^n$ or $v\in \C^n$: $|\,.\,|$ will be the supremum norm, $|\,.\,|_1$ the $\ell_{1}$-norm and $\Vert\,.\,\Vert$ the Euclidean (or Hermitian) norm. 

\paraga Let us now describe the plan of the paper. Section~\ref{sNE2} is devoted to the proof of Theorem~\ref{thelli} and Theorem~\ref{thelli2}. In~\ref{sNE21}, we give a statement of the Birkhoff normal form up to an exponentially small remainder. In~\ref{sNE22}, we will explain how Nekhoroshev's estimates obtained in \cite{BN09} generalize in the neighbourhood of elliptic fixed points, and how they imply Theorem~\ref{thelli2}. In~\ref{sNE23}, we will show how Theorem~\ref{thelli} follows from a simple combination of Birkhoff's estimates and Nekhoroshev's estimates, provided our assumption on $h_\infty$ is satisfied. Then, in section~\ref{sNE3}, we will state similar results for invariant Lagrangian tori and more generally for invariant linearly stable isotropic reducible tori. Finally, an appendix is devoted to our genericity assumptions.

\section{Proof of Theorem~\ref{thelli} and Theorem~\ref{thelli2}}\label{sNE2}

In the sequel, we recall that we will use the ``formal" actions  
\[ \tilde{I}=\tilde{I}(z)=\frac{1}{2}(z_1^2+z_{n+1}^2, \dots, z_n^2+z_{2n}^2) \in \R^n, \]
but one has to remember that these are nothing but notations for expressions in $z\in\R^{2n}$. Moreover, we will also need to use complex coordinates for the normal forms, and, abusing notations, we will also denote them by $z\in\C^{2n}$, but of course the solutions we consider are real.

\subsection{Birkhoff's estimates} \label{sNE21}

Here we consider a Hamiltonian as in~(\ref{Ham1}), and we assume that the vector $\alpha$ is $(\gamma,\tau)$-Diophantine. In this context, the following result is classical.

\begin{proposition}\label{B1}
Under the previous assumptions, if $\rho\MP \gamma$, then there exist an integer $m=m(\rho)$ and an analytic symplectic transformation
\[ \Phi_m: \mathcal{D}_{3s/4} \rightarrow D_{s} \]
such that  
\[ H\circ\Phi_m(z)=h_m(\tilde{I})+f_m(z) \]
is in Birkhoff normal form, with a remainder $f_m$ satisfying the estimate
\[ |f_m|_{3s/4} \MP \rho \exp\left(-(\gamma\rho^{-1})^a\right), \quad a=(1+\tau)^{-1}. \]
Moreover, $|\Phi_m-\mathrm{Id}|_{3s/4} \MP \gamma^{-1}\rho$ and the image of $\Phi_m$ contains the domain $\mathcal{D}_{s/2}$. 
\end{proposition}

For a proof, we refer to \cite{GDFGS} and \cite{DG96b}. The analogous result for invariant Lagrangian tori can be found in \cite{PW94} or \cite{Pos93}, and in \cite{JV97} in the more general case of isotropic and reducible linearly stable invariant tori. 

In the above proposition, one has to choose the integer $m=m(\rho)$ of order $(\gamma \rho^{-1})^{(\tau+1)^{-1}}$. So letting $\rho$ go to zero, the degree of the polynomial $h_m$ goes to infinity, and this explains why in the proof of Theorem~\ref{thelli} we require a condition on the whole formal power series $h_{\infty}$.

\subsection{Nekhoroshev's estimates and proof of Theorem~\ref{thelli2}}\label{sNE22}

\paraga Here we consider the Hamiltonian 
\begin{equation}\label{HamN}
\begin{cases}
H(z)=h(\tilde{I})+f(z) \\
H \in \mathcal{A}_s, \; h\in SDM_{\gamma'}^{\tau'},  \; |f|_s<\varepsilon 
\end{cases} \tag{$E$}
\end{equation}
and we have assumed that, on the real part of the domain, the derivatives up to order $3$ of $h$ are uniformly bounded by some constant $M>1$. The definition of the set $SDM_{\gamma'}^{\tau'}$ is recalled in Appendix~\ref{generic}. 

\begin{theorem}\label{N1}
Let $H$ be as in (\ref{HamN}), with  $\tau'\geq 2$ and $\gamma'\leq 1$. Then there exists $\varepsilon_0$ such that if $\varepsilon\leq\varepsilon_0$, for every solution $z(t)$ with $|\tilde{I}(0)|<1$, we have
\[ |\tilde{I}(t)-\tilde{I}(0)| < (n^2+1)\varepsilon^{b'}, \quad |t| < \exp(\varepsilon^{-a'}), \]
with the exponents $a'=b'=3^{-1}(2(n+1)\tau')^{-n}$.
\end{theorem}

Theorem~\ref{thelli2} is now an immediate consequence of this result and Theorem~\ref{genericquadratic} (see Appendix~\ref{generic}). 

\begin{proof}[Proof of Theorem~\ref{thelli2}]
From Theorem~\ref{genericquadratic}, we know that for almost all $\beta\in S_n(\R)$, the Hamiltonian $h(\tilde{I})=\alpha.\tilde{I}+\beta \tilde{I}.\tilde{I}$ belongs to $SDM^{\tau'}(B)$ with $\tau'>n^2+1$. So we can apply Theorem~\ref{N1}: for every solution $z(t)$ with $|\tilde{I}(0)|<1$, we have
\[ |\tilde{I}(t)-\tilde{I}(0)| < (n^2+1)\varepsilon^{b'}, \quad |t| < \exp(\varepsilon^{-a'}), \]
with the exponents $a'=b'=3^{-1}(2(n+1)\tau')^{-n}$. In particular, this gives  
\[|\tilde{I}(t)-\tilde{I}(0)|_1<n(n^2+1)\varepsilon^{b'},\quad |t|<\exp(\varepsilon^{-a'}), \]
for every solution $z(t)$ with $|\tilde{I}(0)|_1<1$.
\end{proof}

\paraga The statement of Theorem~\ref{N1} is the analogue of the main statement of \cite{BN09}. However, the difference is that here we are using Cartesian coordinates and not action-angle coordinates ({\it i.e.} symplectic polar coordinates), and we \emph{cannot} use the latter since they become singular at the origin. So we \emph{cannot} apply directly the main result of \cite{BN09}. This is not a serious issue when applying KAM theory in this context (see \cite{Arn63b} or \cite{Pos82} for example), but this becomes problematic in Nekhoroshev theory (see \cite{Loc92} or \cite{Loc95} for detailed explanations). This result was only conjectured by Nekhoroshev in \cite{Nek77}, and it took a long time before it could be solved in the convex case (\cite{Nie98},\cite{BFG98}). Here we are able to solve this problem in the generic case. The reason is that even though we cannot apply the result of \cite{BN09}, we can use exactly the same approach, since the method of averagings along unperturbed periodic flows is intrinsic, {\it i.e.} independent of the choice of coordinates, a fact that was first used implicitly in \cite{Nie98} and made completely clear in \cite{Pos99a}. 

The proof of such estimates usually requires an analytic part, which boils down to the construction of suitable normal forms, and a geometric part. The geometric part of \cite{BN09} goes exactly the same way, so in the sequel we will restrict ourselves to indicating the very slight modifications in the construction of the normal forms.

\paraga Consider linearly independent periodic vectors $\omega_1,\dots,\omega_n$ of $\R^n$, with periods $(T_1,\dots,T_n)$, that is 
\[ T_j=\inf\{t>0 \; | \; t\omega_i\in\Z^n\}, \quad 1\leq j \leq n. \]
Define the domains 
\[ \mathcal{D}_{r_j,s_j}(\omega_j)=\{z \in \mathcal{D}_{s_j} \; | \; |\nabla h(\tilde{I})-\omega_j|\MP r_j\}, \quad 1\leq j \leq n, \] 
given two sequences $(r_1,\dots,r_n)$ and $(s_1,\dots,s_n)$ (recall that $\mathcal{D}_{s_j}$ is the complex ball of radius $s_j$). We will denote by $l_j$ the linear Hamiltonian with frequency $\omega_j$, that is $l_j(\tilde{I})=\omega_j.\tilde{I}$.

The supremum norm of a function $f$ defined on $\mathcal{D}_{r_j,s_j}(\omega_j)$ will be simply denoted by 
\[ |f|_{r_j,s_j}=|f|_{\mathcal{D}_{r_j,s_j}(\omega_j)}, \]
and for a Hamiltonian vector field $X_f$, we will write 
\[ |X_f|_{r_j,s_j}=\max_{1 \leq i \leq 2n}|\partial_{z_i}f|_{r_j,s_j}. \]
To obtain normal forms on these domains we will make the following assumptions $(A_j)$, for $j\in\{1,\dots,n\}$, where $(A_1)$ is
\begin{equation} 
\begin{cases}
mT_1\varepsilon\PM r_1, \; mT_1r_1\PM s_1, \; 0<r_1 \MP s_1, \\
\mathcal{D}_{r_1,s_1}(\omega_1)\neq \emptyset , \; s_1 \PM s, 
\end{cases} \tag{$A_1$}
\end{equation} 
and for $j\in\{2,\dots,n\}$, ($A_j)$ is
\begin{equation} 
\begin{cases}
mT_j\varepsilon\PM r_j, \; mT_jr_j\PM s_j, \; 0<r_j \MP s_j,\\
\mathcal{D}_{r_j,s_j}(\omega_j)\neq \emptyset,\; \mathcal{D}_{r_j,s_j}(\omega_{j})\subseteq \mathcal{D}_{2r_{j-1}/3,2s_{j-1}/3}(\omega_{j-1}).  
\end{cases} \tag{$A_j$}
\end{equation}

With these assumptions, one can prove the following proposition.

\begin{proposition} 
Consider $H=h+f$ on the domain $\mathcal{D}_{r_1,s_1}(\omega_1)$, with $|X_f|_{r_1,s_1} < \varepsilon$, and let $j\in\{1,\dots,n\}$. For any $i\in\{1,\dots,j\}$, if $(A_i)$ is satisfied,  then there exists an analytic symplectic transformation 
\[ \Psi_j: \mathcal{D}_{2r_j/3,2s_j/3}(\omega_j) \rightarrow \mathcal{D}_{r_1,s_1}(\omega_1)\] 
such that
\begin{equation*}
H \circ \Psi_j=h+g_j+f_j,
\end{equation*}
with $\{g_j,l_i\}=0$ for $i \in \{1, \dots, j\}$, and the estimates
\begin{equation*}
|X_{g_j}|_{2r_j/3,2s_j/3} \MP \varepsilon, \quad |X_{f_j}|_{2r_j/3,2s_j/3} \MP e^{-m} \varepsilon. 
\end{equation*}
Moreover, we have $\Psi_j=\Phi_1 \circ \cdots \circ \Phi_j$ with
\[ \Phi_i: \mathcal{D}_{2r_i/3,2s_i/3}(\omega_i) \rightarrow \mathcal{D}_{r_i,s_i}(\omega_i)\] 
such that $|\Phi_i-\mathrm{Id}|_{2r_i/3,2s_i/3} \PM r_i$.
\end{proposition}

The proof is completely analogous to the corresponding one in~\cite{BN09}, Appendix A, to which we refer for more details. In fact, here the proof is even simpler since one does not have to use ``weighted" norms for vector fields. It relies on a finite composition of averagings along the periodic flows generated by $l_j$, $j\in\{1,\dots,n\}$. The case $j=1$ is due to Pöschel (\cite{Pos99a}) and, for $j>1$, the proof goes by induction using our assumption $(A_j)$, $j\in\{1,\dots,n\}$.

Once we have this normal form, the rest of the proof in~\cite{BN09} goes exactly the same way: every solution $z(t)$ of $H$ with $|\tilde{I}(0)|<1$ satisfies
\[ |\tilde{I}(t)-\tilde{I}(0)| < (n^2+1)\varepsilon^{b'}, \quad |t| < \exp(\varepsilon^{-a'}), \]
provided that $\varepsilon\leq\varepsilon_0$, with $\varepsilon_0$ depending on $n,s,M,\gamma'$ and $\tau'$ and with the exponents
\[ a'=b'=3^{-1}(2(n+1)\tau')^{-n}.\] 

\subsection{Proof of Theorem~\ref{thelli}}\label{sNE23}

Now we can finally prove Theorem~\ref{thelli}, by using successively Birkhoff's estimates and Nekhoroshev's estimates.

\begin{proof}[Proof of Theorem~\ref{thelli}]
Let $H$ be as in~(\ref{Ham1}), first assume that $\rho<\rho_1$ with $\rho_1 \EP \gamma$ so that using our assumption $(G_1)$ we can apply Proposition~\ref{B1}: there exist an integer $m=m(\rho)$ and an analytic symplectic transformation
\[ \Phi_m: \mathcal{D}_{3s/4} \rightarrow \mathcal{D}_s \]
such that 
\[ H\circ\Phi_m(z)=h_m(\tilde{I})+f_m(z) \]
is in Birkhoff normal form, with a remainder $f_m$ satisfying the estimate
\[ |f_m|_{3s/4} \MP \rho \exp\left(-(\gamma\rho^{-1})^{a^{-1}}\right), \quad a=\tau+1. \]
So let $H_m=H\circ\Phi_m$, and set 
\[\varepsilon = \rho \exp\left(-(\gamma\rho^{-1})^{a^{-1}}\right).\] 
By our assumption $(G_m)$, for $m\geq 2$, the Hamiltonian $H_m$, which is defined on the domain $\mathcal{D}_{3s/4},$ satisfies~(\ref{HamN}). Now assume that $\varepsilon \MP \varepsilon_0$, which gives another threshold $\rho<\rho_{2}$, with $\rho_2$ also depending on $\gamma'$, and our final threshold is $\rho_0=\min\{\rho_1,\rho_2\}$. We can apply Theorem~\ref{N1}: every solution $z_m(t)$ of $H_m$ with $|\tilde{I}_m(0)|_1<1$ satisfies
\[ |\tilde{I}_m(t)-\tilde{I}_m(0)|_1 \MP \varepsilon^{b'}, \quad |t| < \exp(\varepsilon^{-a'}), \]
with 
\[ a'=b'=3^{-1}(2(n+1)\tau')^{-n}.\] 
Recalling the definition of $\varepsilon$, this gives
\[ |\tilde{I}_m(t)-\tilde{I}_m(0)|_1 \MP \rho^{b'} \exp\left(-b'(\gamma\rho^{-1})^a\right), \quad |t| < \exp\left(\rho^{-a'}  \exp\left(a'(\gamma\rho^{-1})^a\right)\right). \]
However, one has
\[ \rho^{b'} \exp\left(-b'(\gamma\rho^{-1})^a\right) < \gamma^{-1}\rho, \]
and as $\Phi_m$ satisfies $|\Phi_m-\mathrm{Id}|_{3s/4}\MP\gamma^{-1}\rho$, and its image contains the domain $\mathcal{D}_{s/2}$, a standard argument gives
\[ |I(t)-I(0)|_1 \MP \gamma^{-1}\rho, \quad |t| < \exp\left(\rho^{-a'}  \exp\left(a'(\gamma\rho^{-1})^a\right)\right), \] 
for any solution $z(t)$ of $H$ with $|\tilde{I}(0)|_1<1$. 
\end{proof}

\section{Further results and comments}\label{sNE3}

As we have already mentioned, the idea of combining both Birkhoff theory and Nekhoroshev theory to obtain super-exponential stability was discovered by Morbidelli and Giorgilli (\cite{MG95}) in the context of Lagrangian Diophantine tori. Evidently, we can also state results in this context.

\paraga Consider a Hamiltonian system on a manifold which carries an invariant Lagrangian Diophantine torus, that is an invariant sub-manifold $\mathcal{T}$ which is diffeomorphic to the standard torus $\T^n$, and whose induced flow is conjugated to a linear flow on $\T^n$ with a Diophantine frequency. Since the torus is Lagrangian, one can locally reduce the situation to a Hamiltonian defined on $T^*\T^n=\T^n\times\R^n$, having $\T^n\times\{0\}$ as the invariant torus. Moreover, by invariance and transitivity of the torus, in the coordinates $(\theta,I)\in\T^n\times\R^n$ we can write
\begin{equation*}
H(\theta,I)=\omega.I+F(\theta,I),
\end{equation*} 
where $\omega$ is a $(\gamma,\tau)$-Diophantine vector and $F(\theta,I)=O(|I|^2)$. After some scalings one is led to consider
\begin{equation}\label{Ham3}
\begin{cases}
H(\theta,I)=\omega.I+f(\theta,I) \\
H \in \mathcal{A}_{s}, |f|_{s} < \rho
\end{cases}\tag{$C$} 
\end{equation} 
where $\mathcal{A}_{s}$ is the space of holomorphic functions on the domain 
\[ \mathcal{D}_{s}=\{(\theta,I)\in(\C^n/\Z^n)\times \C^{n} \; | \; |\mathcal{I}(\theta)|<s,\;|I|< s\}, \] 
with $\mathcal{I}(\theta)$ the imaginary part of $\theta$. Here one can also define polynomials $h_m$ and a formal power series $h_\infty$ and we can state the following result.

\begin{theorem}\label{thlag}
Suppose $H$ is as in~(\ref{Ham3}). Then, under a generic condition on $h_\infty$, there exist positive constants $a,a',c_1,c_2$ and $\rho_0$ such that for $\rho\leq\rho_0$, every solution $(\theta(t),I(t))$ of $H$ with $|I(0)|<1$ satisfies
\[ |I(t)-I(0)|< c_1\rho,\quad |t| < \exp\left(\rho^{-a'}\exp(c_2a'\rho^{-a})\right).  \]
\end{theorem}

The assumption on $h_\infty$ and the values of the constants $a$ and $a'$ are the same as in Theorem~\ref{thelli}, as the proof is completely analogous. In fact, it is even simpler since we are using action-angle coordinates, and therefore we can immediately use Nekhoroshev's estimates obtained in~\cite{BN09} without any modifications. 

However, it is important to note that one cannot obtain a statement similar to Theorem~\ref{thelli2}, simply because in this case a non-resonant condition up to a finite order does not allow to build the corresponding Birkhoff normal form. 

If we compare this result with~\cite{MG95}, our assumption is generic and we do not require any convexity. But of course the price to pay is that one has to consider the full set of Birkhoff invariants.

\paraga As a final result, one can also obtain similar estimates for the general case of a linearly stable lower-dimensional torus, under the common assumptions of isotropicity and reducibility (which were automatic for a fixed point or a Lagrangian torus). In that context, it is enough to consider a Hamiltonian defined in $\T^k \times \R^k \times \R^{2l}$ (by isotropicity), of the form
\[ H(\theta,I,z)= \omega.I+\frac{1}{2}Bz.z+F(\theta,I,z).\]
Here $B$ is a symmetric matrix (constant by reducibility) such that $J_{2l}B$ has a purely imaginary spectrum ($J_{2l}$ being the canonical symplectic structure of $\R^{2l}$), and $F(\theta,I,z)=O(|I|^{2},||z||^3)$. In those coordinates, the invariant torus is simply given by $I=0$, $z=0$, and this generalizes both the case of an elliptic fixed point (where the directions $(\theta,I)$ are absent) and of a Lagrangian invariant torus (where the directions $z$ are absent). If the spectrum $\{\pm i\alpha_1, \dots, \pm i\alpha_l\}$ of $J_{2l}B$ is simple, one can assume further that
\[ H(\theta,I,z)= \omega.I+\alpha.\tilde{I}+F(\theta,I,z),\]
where $\tilde{I}$ are the ``formal actions" associated to the $z$ variables. Therefore, after some appropriate scalings we can consider
\begin{equation}\label{Ham4}
\begin{cases}
H(\theta,I,z)=\omega.I+\alpha.\tilde{I}+F(\theta,I,z) \\
H \in \mathcal{A}_{s}, |f|_{s} < \rho
\end{cases}\tag{$D$} 
\end{equation} 
where $\mathcal{A}_{s}$ is the space of holomorphic functions on the domain 
\[ \mathcal{D}_{s}=\{(\theta,I,z)\in(\C^k/\Z^k)\times \C^{k}\times\C^{2l} \; | \; |\mathcal{I}(\theta)|<s,\;|I|< s,\;||z||<s\}. \] 
Under a suitable Diophantine condition on the vector $(\omega,\alpha)\in\R^{k+l}$, one can define polynomials $h_m$ and a formal series $h_\infty$ depending on $J=(I,\tilde{I})$. Birkhoff's exponential estimates in this more difficult situation have been obtained in \cite{JV97}. Regarding Nekhoroshev's estimates for a generic integrable Hamiltonian which depends both on actions and formal actions, they can be easily obtained by obvious modifications of our method. Therefore we can state the following result.

\begin{theorem}\label{thellilag}
Suppose $H$ is as in~(\ref{Ham4}). Then under a generic condition on $h_\infty$, there exist positive constants $a,a',c_1,c_2$ and $\rho_0$ such that for $\rho\leq\rho_0$, every solution $(\theta(t),I(t),z(t))$ of $H$ with $|J(0)|<1$ satisfies
\[ |J(t)-J(0)|< c_1\rho,\quad |t| < \exp\left(\rho^{-a'}\exp(c_2a'\rho^{-a})\right).  \]
\end{theorem} 

Once again, the condition on $h_\infty$ and the values of the exponents are the same.

\paraga Let us add that one could easily give similar estimates in the discrete case, that is for exact symplectic diffeomorphisms near an elliptic fixed point, an invariant Lagrangian torus or an invariant linearly stable isotropic reducible torus. Even if one has the possibility to re-write the proof in these settings, the easiest way is to use suspension arguments, as it is done qualitatively in~\cite{Dou88} or quantitatively in \cite{KP94} (see also \cite{PT97} for a different approach) and deduce stability results in the discrete case from the corresponding results in the continuous case.  

To conclude, let us mention that important examples of invariant tori satisfying our assumptions (linearly stable, reducible, isotropic) are those given by KAM theory. However, the latter not only gives individual tori but a whole Cantor family (see \cite{Pos01} or \cite{AKN97}). In this context, Popov has proved exponential stability estimates for the family of Lagrangian KAM tori, if the Hamiltonian is analytic or Gevrey (\cite{Pop00} and \cite{Pop04}). His proof relies on a KAM theorem with Gevrey smoothness on the parameters (in the sense of Whitney) and some kind of simultaneous Birkhoff normal form over the Cantor set of tori. We believe that our method should be useful in trying to extend those results to obtain super-exponential stability under generic conditions. But clearly this is a more difficult problem, and the first step is to obtain Nekhoroshev's estimates in Gevrey regularity for a generic integrable Hamiltonian, the quasi-convex case having been settled in~\cite{MS02}.

\appendix

\section{Generic assumptions} \label{generic}

In this appendix, we will show that our assumption $(G)$ is generic, in the sense that it defines a prevalent set in the infinite dimensional space of formal power series.  

\paraga Let us first recall the definition of Simultaneous Diophantine Morse functions (SDM in the following). Let $G(n,k)$ be the set of all vector subspaces of $\R^n$ of dimension $k$. We equip $\R^n$ with the Euclidean scalar product, and given an integer $L\in\N^*$, we define $G^{L}(n,k)$ as the subset of $G(n,k)$ consisting of subspaces whose orthogonal complement can be generated by integer vectors with components bounded by $L$. In the sequel, $B$ will be an arbitrary open ball of $\R^n$.   

\begin{definition} \label{sdm}
A smooth function $h: B \rightarrow \R$ is said to be SDM if there exist $\gamma'>0$ and $\tau' \geq 0$ such that for any $L\in\N^*$, any $k\in\{1,\dots,n\}$ and any $\Lambda\in G^{L}(n,k)$, there exists $\left(e_1,\dots,e_k\right)$ (resp. $\left(f_1,\ldots,f_{n-k}\right)$), an orthonormal basis of $\Lambda$ (resp. of $\Lambda^\perp$), such that the function $h_\Lambda$ defined on $B$ by 
\[ h_\Lambda(\alpha,\beta)=h\left(\alpha_1 e_1+\dots+\alpha_k e_k+\beta_1f_1+\dots+\beta_{n-k} f_{n-k}\right) \]
satisfies the following: for any $(\alpha,\beta) \in B$, 
\[ \Vert\partial_\alpha h_\Lambda(\alpha,\beta)\Vert \leq \gamma' L^{-\tau'} \Longrightarrow \Vert \partial_{\alpha\alpha} h_\Lambda(\alpha,\beta).\eta\Vert>\gamma' L^{-\tau'}\Vert\eta\Vert, \]
for any $\eta \in \R^k\setminus\{0\}$.
\end{definition} 

This definition is inspired by the steepness condition of Nekhoroshev and the quantitative Morse-Sard theory of Yomdin (see \cite{BN09} for more explanations). It depends on a choice of coordinates adapted to the orthogonal decomposition $\Lambda\oplus\Lambda^{\perp}$, so for $\Lambda\in G^{L}(n,k)$ and $(\alpha,\beta)\in B$, $\partial_\alpha h_\Lambda(\alpha,\beta)$ is a vector in $\R^k$ and $\partial_{\alpha\alpha} h_\Lambda(\alpha,\beta)$ is a symmetric matrix of size $k$ with real entries. 

\begin{remark}\label{remsdm}
Note also that the definition can be stated as the following alternative: for any $(\alpha,\beta) \in B$, either we have $\Vert\partial_\alpha h_\Lambda(\alpha,\beta)\Vert>\gamma L^{-\tau}$ or $\Vert\partial_{\alpha\alpha} h_\Lambda(\alpha,\beta).\eta\Vert>\gamma L^{-\tau}\Vert\eta\Vert$ for any $\eta \in \R^k\setminus\{0\}$. Hence for a given function it is sufficient to verify that $\Vert \partial_{\alpha\alpha} h_\Lambda(\alpha,\beta).\eta\Vert>\gamma L^{-\tau}\Vert\eta\Vert$ for any $\eta \in \R^k\setminus\{0\}$, and we will use this fact later (in Theorem~\ref{genericquadratic}).
\end{remark}

\paraga The set of SDM functions on $B$ with respect to $\gamma'>0$ and $\tau' \geq 0$ will be denoted by $SDM_{\gamma'}^{\tau'}(B)$, and we will also use the notation 
\[ SDM^{\tau'}(B)=\bigcup_{\gamma'>0}SDM_{\gamma'}^{\tau'}(B). \]
The following theorem was proved in \cite{BN09}, and it relies on non trivial results from quantitative Morse-Sard theory (\cite{Yom83},\cite{YC04}).

\begin{proposition}[\cite{BN09}]\label{propreva}
Let $\tau > 2(n^2+1)$, and $h\in C^{2n+2}(B)$. Then for Lebesgue almost all $\xi\in\R^n$, the function $h_\xi$, defined by $h_\xi(I)=h(I)-\xi.I$ for $I\in B$, belongs to $SDM^{\tau'}(B)$.      
\end{proposition}

Now let us recall the definition of a prevalent set (\cite{HSY92}, see also \cite{OY05}). 

\begin{definition}
Let $E$ be a completely metrizable topological vector space. A Borel subset $S \subseteq E$ is said to be shy if there exists a Borel measure $\mu$ on $E$, with $0<\mu(C)<\infty$ for some compact set $C\subseteq E$, such that $\mu(x+S)=0$ for all $x\in E$. 

An arbitrary set is called shy if it is contained in a shy Borel subset, and finally the complement of a shy set is called prevalent.
\end{definition}  

For a finite dimensional vector space $E$, by an easy application of Fubini theorem, prevalence is equivalent to full Lebesgue measure. The following ``genericity" properties can be checked (\cite{OY05}): a prevalent set is dense, a set containing a prevalent set is also prevalent, and prevalent sets are stable under translation and countable intersection. Furthermore, we have an easy but useful criterion for a set to be prevalent.

\begin{proposition}[\cite{HSY92}]\label{proppreva}
Let $A$ be a Borel subset of $E$. Suppose there exists a finite dimensional subspace $F$ of $E$ such that, denoting $\lambda_F$ the Lebesgue measure supported on $F$, the set $x+A$ has full $\lambda_{F}$-measure for all $x\in E$. Then $A$ is prevalent.
\end{proposition}

It is an obvious consequence of Proposition~\ref{propreva} and Proposition~\ref{proppreva} that $SDM^{\tau'}(B)$ is prevalent in $C^{2n+2}(B)$ for $\tau' > 2(n^2+1)$.

\paraga Now let $\mathcal{P}_\infty=\R[[X_1,\dots,X_n]]$ be the space of all formal power series in $n$ variables with real coefficients. It is naturally a Fréchet space, as the projective limit of the finite dimensional spaces $\mathcal{P}_m$ consisting of polynomials in $n$ variables of degree less than or equal to $m$. We define the subset 
\[ \mathcal{S}_{\infty}^{\tau'}=\{ h_\infty \in\mathcal{P}_\infty \; | \; h_m \in SDM^{\tau'}(B),\; \forall m\geq 2  \}, \]
where $h_m=\sum_{k=1}^{m}h^k$ if $h_\infty=\sum_{k\geq 1}h^k$, and we identify the polynomial $h_m$ with the associated function defined on $B$. Let us also define
\[ \mathcal{D}_{\infty}^{\tau}=\{ h_\infty \in\mathcal{P}_\infty \; | \; h_1(X)=\alpha.X, \; \alpha\in \mathcal{D}^{\tau}\}, \]
where $\mathcal{D}^{\tau}$ is the set of Diophantine vectors of $\R^n$ with exponent $\tau$, and finally
\[ \mathcal{G}_{\infty}^{\tau,\tau'}=\mathcal{D}_{\infty}^{\tau}\cap\mathcal{S}_{\infty}^{\tau'}.  \]
The set $\mathcal{G}_{\infty}^{\tau,\tau'}$ is the set of formal power series for which condition $(G)$ holds.

\begin{theorem}\label{genericseries}
For $\tau>n-1$ and $\tau' >2(n^2+1)$, the set $\mathcal{G}_{\infty}^{\tau,\tau'}$ is prevalent in $\mathcal{P}_\infty$.  
\end{theorem}

\begin{proof}
As the intersection of two prevalent sets is prevalent, it is enough to prove that both sets $\mathcal{D}_{\infty}^{\tau}$, for $\tau>n-1$, and $\mathcal{S}_{\infty}^{\tau'}$, for $\tau' >2(n^2+1)$, are prevalent.

For the set $\mathcal{D}_{\infty}^{\tau}$, this is an easy consequence of the fact that $\mathcal{D}^{\tau}$ is of full Lebesgue measure in $\R^n$, for $\tau>n-1$, and Proposition~\ref{proppreva} with $F=\mathcal{P}_1$, the space of linear forms. For the set $\mathcal{S}_{\infty}^{\tau'}$, first note that we can write
\[ \mathcal{S}_{\infty}^{\tau'}=\bigcap_{m\geq 2}\mathcal{S}_{\infty,m}^{\tau'}, \]
where, for an integer $m\geq 2$,
\[ \mathcal{S}_{\infty,m}^{\tau'}=\{ h_\infty \in\mathcal{P}_\infty \; | \; h_m \in SDM^{\tau'}(B) \}. \]
As a countable intersection of prevalent sets is prevalent, it is enough to prove that for each $m\geq 2$, the set $\mathcal{S}_{\infty,m}^{\tau'}$ is prevalent in $\mathcal{P}_\infty$. But once again this is just a consequence of Proposition~\ref{propreva} and Proposition~\ref{proppreva} with $F=\mathcal{P}_1$ the space of linear forms.        
\end{proof}

For $m\geq 2$, the set of polynomials $h_m$ for which condition $(G_m)$ is satisfied is given by
\[ \mathcal{S}_{m}^{\tau'}=\{ h_m \in\mathcal{P}_m \; | \; h_m \in SDM^{\tau'}(B) \}, \]
and the proof of the above theorem immediately gives the following result.

\begin{theorem}
For $\tau'>2(n^2+1)$, the set $\mathcal{S}_{m}^{\tau'}$ is of full Lebesgue measure in $\mathcal{P}_m$.  
\end{theorem}

\paraga Now in the special case $m=2$, we can state a refined result which is due to Niederman (\cite{NieU}).

\begin{theorem}\label{genericquadratic}
For Lebesgue almost all $\beta\in S_n(\R)$, the function 
\[h(I)=\alpha.I+\beta I.I\] 
belongs to $SDM^{\tau'}(B)$ provided $\tau'>n^2+1$. 
\end{theorem} 

In the above theorem, there is no condition on $\alpha$, and contrary to Proposition~\ref{propreva}, the proof does not rely on Morse-Sard theory. Let us denote by $\lambda$ the one-dimensional Lebesgue measure and by $I_k$ the identity matrix of size $k$. We shall use the following elementary lemma. 

\begin{lemma}\label{lemmequa}
Let $k\in\{1,\dots,n\}$, $\beta_k\in S_k(\R)$ and $\kappa>0$. Then there exists a subset $\mathcal{C}_\kappa \subseteq \R$ such that
\[ \lambda(\mathcal{C}_\kappa) \leq 2k \kappa, \]
and for any $\xi\notin\mathcal{C}_\kappa$, the matrix $\beta_{k,\xi}=\beta_k-\xi I_k$ satisfies
\[ \Vert\beta_{k,\xi}.\eta\Vert>\kappa\Vert\eta\Vert, \]
for any $\eta\in\R^k\setminus\{0\}$.
\end{lemma}

Of course, the set $\mathcal{C}_\kappa$ depends on the matrix $\beta_k$.

\begin{proof}
Let $\{\lambda_1,\dots,\lambda_k\}$ be the eigenvalues of $\beta_k$, then in an orthonormal basis of eigenvectors for $\beta_k$, the matrix $\beta_{k,\xi}$ is also diagonal, with eigenvalues $\{\lambda_1-\xi,\dots,\lambda_k-\xi\}$. Then one has $\Vert\beta_{k,\xi}.\eta\Vert>\kappa\Vert\eta\Vert$ for any $\eta\in\R^k\setminus\{0\}$ provided that for all $i\in\{1,\dots,k\}$, $|\lambda_i-\xi|>\kappa$, that is if $\xi$ does not belong to
\[ \mathcal{C}_\kappa=\bigcup_{i=1}^{k}[\lambda_i-\kappa,\lambda_i+\kappa]. \]
The measure estimate $\lambda(\mathcal{C}_\kappa) \leq 2k \kappa$ is trivial.  
\end{proof}

With this lemma, the proof is now similar to that of Proposition~\ref{propreva}.

\begin{proof}[Proof of Theorem~\ref{genericquadratic}]
Let $h(I)=\alpha.I+\beta I.I$, and given $\Lambda\in G^{L}(n,k)$, we denote by $\beta_\Lambda\in S_k(\R)$ the matrix which represents the quadratic form $\beta I.I$ restricted to the subspace $\Lambda$. Since the second derivative of $h$ along any subspace is constant, then coming back to definition~\ref{sdm} and using remark~\ref{remsdm}, $h\in SDM_{\gamma'}^{\tau'}$ if
\begin{equation}
\Vert\beta_\Lambda.\eta\Vert>\gamma'L^{-\tau'}\Vert\eta\Vert, \label{genqua}
\end{equation}  
for any $\Lambda\in G^{L}(n,k)$ and any $\eta\in\R^k\setminus\{0\}$. Let $A_{\gamma'}^{\tau'}$ be the subset of $S_n(\R)$ whose elements contradict condition~(\ref{genqua}), that is
\[ A_{\gamma'}^{\tau'}=\{\beta\in S_n(\R) \; | \; \Vert\beta_{\Lambda}.\eta\Vert\leq \gamma' L^{\tau'}\Vert\eta\Vert, \; \Lambda\in G^{L}(n,k), \; \eta\in\R^k\setminus\{0\} \}, \]
and 
\[ A^{\tau'}=\bigcap_{\gamma'>0}A_{\gamma'}^{\tau'}.\] 
What we need to show is that $A^{\tau'}$ has zero Lebesgue measure in $S_n(\R)$ provided $\tau'>n^2+1$. 

Apply lemma~\ref{lemmequa} to $\beta_\Lambda\in S_k(\R)$, with $\kappa=\gamma'L^{-\tau'}$, to have a subset $\mathcal{C}_{\gamma',\tau',L,\Lambda} \subseteq \R$ such that
\begin{equation}\label{estimqua}
\lambda(\mathcal{C}_{\gamma',\tau',L,\Lambda}) \leq 2k \gamma'L^{-\tau'},
\end{equation}
and for any $\xi\notin\mathcal{C}_{\gamma',\tau',L,\Lambda}$, the matrix $\beta_{\Lambda,\xi}=\beta_\Lambda-\xi I_k$ satisfies
\[ \Vert\beta_{\Lambda,\xi}.\eta\Vert>\gamma'L^{-\tau'}\Vert\eta\Vert \]
for any $\eta\in\R^k\setminus\{0\}$. If we define 
\[ \mathcal{C}_{\gamma',\tau'}=\bigcup_{L \in \N^*}\bigcup_{k\in\{1,\dots,n\}}\bigcup_{\Lambda \in G^L(n,k)}\mathcal{C}_{\gamma',\tau',L,\Lambda}, \]  
then
\[ \mathcal{C}_{\gamma',\tau'}=\{ \xi \in \R \; | \; \beta_\xi=\beta-\xi I_n \in A_{\gamma'}^{\tau'} \} \]
and so
\[ \mathcal{C}_{\tau'}=\bigcap_{\gamma'>0}\mathcal{C}_{\gamma',\tau'}=\{ \xi \in \R \; | \; \beta_\xi \in A^{\tau'} \}. \]
It remains to prove that the Lebesgue measure of $\mathcal{C}_{\tau'}$ is zero, since by Fubini theorem, this will imply that the Lebesgue measure of $A^{\tau'}$ is zero. By our estimate~(\ref{estimqua}), we have
\begin{eqnarray*}
\lambda(\mathcal{C}_{\gamma',\tau'})& \leq &\sum_{L\in\N^*}\sum_{k=1}^{n}|G^L(n,k)|2k \gamma'L^{-\tau'} \\
& \leq & \sum_{L\in\N^*}\sum_{k=1}^{n}L^{n^2}2k \gamma'L^{-\tau'} \\
& = & 2\left(\sum_{k=1}^{n}k\right)\left(\sum_{L\in\N^*}L^{n^2-\tau'}\right)\gamma' \\
& = & n(n+1)\left(\sum_{L\in\N^*}L^{n^2-\tau'}\right)\gamma'
\end{eqnarray*}  
and, since $\tau'>n^2+1$, the above series is convergent. Hence
\[ \lambda(\mathcal{C}_{\tau'})=\inf_{\gamma' >0}\lambda(\mathcal{C}_{\gamma',\tau'})=0. \]

\end{proof}
 
{\it Acknowledgments.} 

The author thanks Laurent Niederman for many useful and lengthy discussions on Nekhoroshev theory, Jacques Féjoz for a careful reading, Jean-Pierre Marco for his support and all the members of ASD team at ``Observatoire de Paris". The author is indebted to the referee for a very careful reading, and for pointing out some mistakes in a previous version.

\addcontentsline{toc}{section}{References}
\bibliographystyle{amsalpha}
\bibliography{superexp2}
\end{document}